\RequirePackage{fix-cm}
\documentclass[smallextended]{svjour3}       
\smartqed  
\usepackage{graphicx}
\usepackage{amsmath,amssymb,epsfig,color,wasysym}
\usepackage[dvipsnames]{xcolor}
\newcommand{\gk}{\gamma_{all,k}^\infty}
\newcommand{\gtwo}{\gamma_{all,2}^\infty}

\begin{document}

\title{Eternal distance-$k$ domination on graphs
}


\author{D. Cox        \and
        E. Meger $^*$ \and 
        M.E. Messinger 
}


\institute{E. Meger \at
Laboratoire d'alg\`ebre, du combinatoire et d'informatique math\'ematique, Universit\'e du Qu\'ebec \`a Montr\'eal\\
              \email{erin.k.meger@ryerson.ca}           
           \and
           D. Cox \at
              Department of Mathematics and Statistics, Mount Saint Vincent University
              \and
              M.E. Messinger \at
              Department of Mathematics and Computer Science, Mount Allison University
}

\date{Nov 16, 2022}

\maketitle

\begin{abstract}
 Eternal domination is a dynamic process by which a graph is protected from an infinite sequence of vertex intrusions. In eternal distance-$k$ domination, guards initially occupy the vertices of a distance-$k$ dominating set.  After a vertex is attacked, guards ``defend'' by each moving up to distance $k$ to form a distance-$k$ dominating set, such that some guard occupies the attacked vertex. The eternal distance-$k$ domination number of a graph is the minimum number of guards needed to defend against any sequence of attacks.  The process is well-studied for the situation where $k=1$. We introduce eternal distance-$k$ domination for $k > 1$. 
    
    Determining whether a given set is an eternal distance-$k$ domination set is in EXP, and in this paper we provide a number of results for paths and cycles, and relate this parameter to graph powers and domination in general. For trees we use decomposition arguments to bound the eternal distance-$k$ domination numbers, and solve the problem entirely in the case of perfect $m$-ary trees.

\noindent \textbf{Funding} M.E. Messinger acknowledges research support from NSERC (grant application 2018-04059).  D. Cox acknowledges research support from NSERC (2017-04401) and Mount Saint Vincent University.  E. Meger acknowledges research support from Universit\'e du Qu\'ebec \`a Montr\'eal and Mount Allison University.


\end{abstract}

\section{Introduction} \label{sec:intro}

In recent years, researchers have become increasingly interested in dynamic domination processes on graphs and {\it eternal domination} problems on graphs have been particularly well-studied (see the survey~\cite{survey}, for example).  In the all-guards move model for eternal domination, a set of vertices are occupied by ``guards" and the vertices occupied by guards form a dominating set on a graph.  At each step, an unoccupied vertex is ``attacked'', by some unknown, external adversary, and guards seek to defend the attacked vertex. Each guard may remain at their current vertex or move along an edge to a neighbouring vertex, in order to occupy a dominating set that contains the attacked vertex. Such a movement of guards is said to ``defend against an attack''. The {\it eternal domination number} of a graph, denoted $\gamma_{all}^\infty$ is the minimum number of guards required to defend against any sequence of attacks, where the subscript and superscript indicate that {\it all} guards can move in response to an attack and the sequence of attacks is infinite.  A generalization of this process, along with other vertex-pursuit games, is the Spy-Game which has been shown to be NP-hard \cite{spygame}.  Given the complexity of determining the eternal domination number of a graph for the all-guards move model, recent work such as~\cite{BKV,finbowetal,FMvB,GHHKM,HKM,LMSS,vBvB}, has focused primarily on bounding or determining the parameter for particular classes of graphs.

In this paper, we extend the notion of eternal domination to that of eternal distance-$k$ domination in the most natural way: suppose at time $t=0$, the guards occupy a set of vertices that form a distance-$k$ dominating set. At each time  step $t>0$ an unoccupied vertex is attacked and every guard moves distance at most $k$ so that the guards occupy the vertices of a distance-$k$ dominating set that contains the attacked vertex. In Figure~\ref{Fig:example1}(a), we see a single guard, $g$ occupying a distance-2 dominating set on $P_5$. When the attack, $a$ occurs at the leaf indicated in grey, the guard moves to the attacked vertex. However, this leaves two vertices indicated in grey that the guard can no longer protect. Not all distance dominating sets are eternal. In Figure~\ref{Fig:example1}(b), we show two guards can eternally distance-2 dominate $P_5$.  We note that for $k=1$, the process is equivalent to the all-guards move model for eternal domination as described above.  Hence, we focus on results for $k \geq 2$.

\begin{figure}
    \centering
    \includegraphics[width=0.41\textwidth]{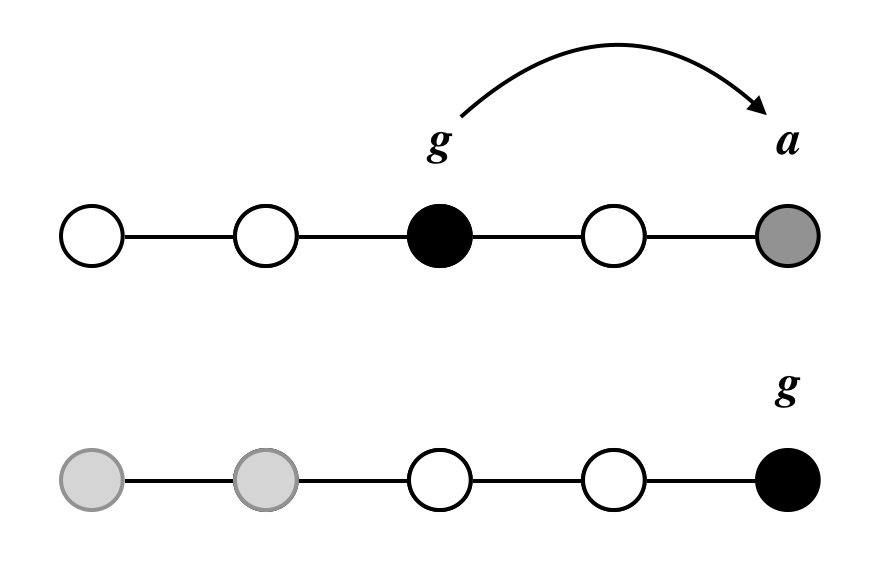}\hspace{1cm}\includegraphics[width=0.4\textwidth]{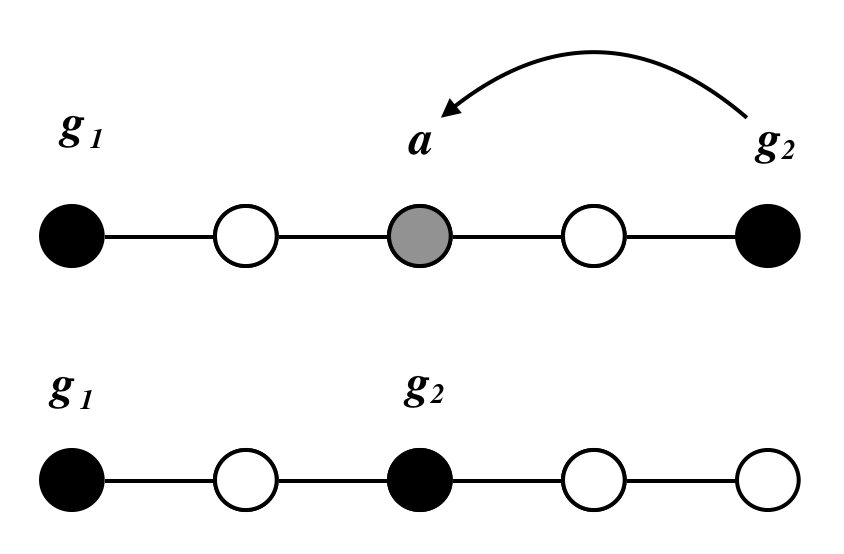}
    \caption{An example of eternal distance-2 domination on $P_5$ with (a) one guard and (b) two guards.}
    \label{Fig:example1}
\end{figure}

We present preliminary results in Section~\ref{sec:2} that provide general bounds, the complexity of the associated decision problem, and determine the eternal distance-$k$ domination number exactly for some small classes of graphs.  In Section~\ref{section_superstar}, we show that the eternal distance-$k$ domination number of a graph is bounded above by the eternal distance-$k$ domination number of a spanning tree of the graph, which motivates us then to focus on trees in Section~\ref{sec:trees}. The eternal domination number of a tree was characterized in~\cite{KlosterM} by using two reductions.  We extend the concepts of these reductions to the eternal distance-$k$ domination model, providing reductions that, informally speaking, ``trim branches" of trees in such a way that the change in the eternal distance-$k$ domination number is controlled.  However, for the eternal distance-$k$ domination model  such reductions are insufficient to fully characterize the eternal distance-$k$ domination number of all trees and we state some resulting open problems.  In Section~\ref{sec:mary}, we determine exactly the eternal distance-$k$ domination number for perfect $m$-ary trees. We follow this with a discussion on decomposing trees for the specific case where $k=2$ in Section~\ref{section_treedecomposition}. Since this paper introduces the concept of eternal distance-$k$ domination, we conclude with a series of open questions in Section~\ref{section_Conclusion}.  

We conclude this section with some formal definitions. As it is often important to differentiate between vertices exactly at distance $k$ and vertices of at most distance $k$, we make use of the following definitions and their corresponding notation. In a graph $G$, the \emph{open distance-$k$ neighbourhood} of $x \in V(G)$ is $N_k(x) = \{ v \in V(G)~:~d(x,v)=k\}$, and the \emph{closed distance-$k$ neighbourhood} of $x \in V(G)$ is $N_k[x] = \{ v \in V(G)~:~d(x,v)\leq k\}$. The {\it eccentricity} of vertex $u$ in graph $G$, denoted $\epsilon_G(u)$, is the maximum distance between $u$ and any other vertex in $G$. More formally, $\epsilon_G(u) = \max_{v \in V(G)} d(u,v)$.

\begin{definition}\label{defn:kdom} Let $G$ be a graph and $k\geq 1$ an integer. A multiset $D=\{d_1,...,d_j\}$ with $d_i \in V(G) \; \forall \;  i$ is a {\bf distance-$k$ dominating set} if every vertex of $V(G) \backslash D$ is at most distance $k$ from a vertex in $D$.  The minimum cardinality of a distance-$k$ dominating multiset in graph $G$ is the {\bf distance-$k$ domination number}, denoted $\gamma_k(G)$.   
\end{definition}

The vertices of the multiset $D$ can be considered the locations of the guards.  We permit multiple guards to occupy the same vertex at the same step.  It was shown in~\cite{finbowetal} that for some graphs, the eternal domination number differs depending on whether multiple guards are permitted to occupy the same vertex at the same step; in particular, allowing multiple guards to occupy the same vertex at the same time can result in a smaller number of guards needed to defend against a sequence of attacks. Though we have defined the distance-$k$ dominating sets to be multisets, we'll informally use the term `set' throughout, acknowledging that guards may occupy the same vertex.

\begin{definition} Let $G$ be a graph.  Let $\mathbb{D}_{k,q}(G)$ be the set of all  distance-$k$ dominating sets of $G$ which have cardinality $q$.  Let $D,D' \in \mathbb{D}_{k,q}(G)$.  We will say $D$  {\bf transforms} to $D'$ if $D=\{v_1,v_2,\dots,v_q\}$, $D'=\{u_1,u_2,\dots,u_q\}$, and $u_i \in N_k[v_i]$ for all $i \in [q]$.\end{definition}

We suppose the multiset $D$ represents the vertices occupied by guards.  If $D$ transforms to $D'$, then the guards occupying vertices of $D$ can each move at most distance $k$ to occupy the vertices of $D'$.  

\begin{definition} \label{def:kdomfamily} An {\bf eternal distance-$k$ dominating family} of $G$ is a subset $\mathcal{E} \subseteq \mathbb{D}_{k,q}(G)$ for some $q$ so that for every $D \in \mathcal{E}$ and each possible attack $v \in V(G)$, there is a distance-$k$ dominating set $D' \in \mathcal{E}$ so that $v \in D'$ and $D$ transforms to $D'$.  

A multiset $D \in \mathbb{D}_{k,q}(G)$ is an {\bf eternal distance-$k$ dominating set} if it is a member of some eternal distance-$k$ dominating family. Eternal domination is a discrete time-process, so at each iteration of an ``attack'' the distance-$k$ dominating set $D$ transforms into some other set $D'$ within the family. 

The {\bf eternal distance-$k$ domination number} of a graph $G$, denoted $\gk(G)$, is the minimum $q$ for which an eternal distance-$k$ dominating family of $G$ exists. We use this notation to indicate that all guards are allowed to move a distance of at most $k$ during each step.

\end{definition}

\section{Eternal distance-$k$ domination on general graphs}\label{sec:2}

\subsection{Preliminary Results}\label{sec:prelim}

In this section, we relate the eternal distance-$k$ domination number to known graph parameters, as well as determine the eternal distance-$k$ domination number for well-known families of graphs. We then determine the complexity of computing this parameter.

By definition, any eternal distance-$k$ dominating set is also a distance-$k$ dominating set, giving us the lower bound in Proposition~\ref{eqn:upper}.  However, we can also bound the eternal distance-$k$ dominating number of a graph by its $\lfloor k/2\rfloor$-domination number:

\begin{proposition}

\label{eqn:upper}For any graph $G$ and integer $k \geq 2$, $$\gamma_k(G)\leq \gk(G) \leq \gamma_{\lfloor\frac{k}{2}\rfloor}(G).$$ \end{proposition}

\begin{proof}
For the upper bound, consider a graph $G$ with $\gamma_{\lfloor k/2 \rfloor}(G)=\ell$ and  $\lfloor k/2 \rfloor$-dominating set $D = \{v_1,v_2,\dots,v_{\ell}\}$. For each $j \in \{1,2,\dots \ell \}$, the maximum distance between any two vertices in $N_{\lfloor k/2 \rfloor}[v_j]$ is $k$.  Specifically, a guard $g_j$ is placed on an arbitrary vertex of $N_{\lfloor k/2 \rfloor}[v_j]$ for each $j \in \{1,2,\dots,\ell \}$.  The guard $g_j$ will only ever occupy vertices of $N_{\lfloor k/2 \rfloor}[v_j]$.  Given an attack at a vertex $x$ in $N_{\lfloor k/2 \rfloor}[v_j]$, guard $g_j$ can move to the attacked vertex and no other guard moves.  We note that it is possible that the attacked vertex is within distance $\lfloor k/2 \rfloor$ from multiple vertices in $D$.
\qed\end{proof}

It is easy to see that if graph $G$ has a universal vertex, then $\gk(G) = \gamma_k(G)=\gamma_{\lfloor\frac{k}{2}\rfloor}(G) = 1$. However, it is worth noting that the difference between $\gamma_k(G)$ and $\gamma_{\lfloor \frac{k}{2}\rfloor}(G)$ can be arbitrarily large. Let $G$ be the graph $K_{1,n}$ where each edge is subdivided $k-1$ times. Then clearly $\gamma_k(G)=1$, but $\gamma_{\lfloor k/2\rfloor}(G)=n+1$. 
  
\begin{theorem}\label{obs:cycle} 
For $n \geq 3$ and $k \geq 1$, \[\gk(C_n) = \gamma_k(C_n) = \Big \lceil \frac{n}{2k+1}\Big \rceil.\]
\end{theorem}

\begin{proof} Observe $C_n$ can be decomposed into $\lceil \frac{n}{2k+1}\rceil$ vertex-disjoint paths, each of length at most $2k$. Since a center vertex of a path of length at most $2k$ will $k$-dominate the path, $\gamma_k(C_n) = \lceil \frac{n}{2k+1}\rceil$.

Next we show that an arbitrary minimum distance-$k$ dominating set is indeed a minimum eternal distance-$k$ dominating set. Place guards on the vertices of an arbitrary minimum distance-$k$ dominating set. Suppose vertex $v$ is attacked and let $u$ be a vertex within distance $k$ of $v$ that contains a guard.  The guard at $u$ moves distance $d(u,v)=x$ to occupy $v$ and all other guards move exactly distance $x$ in the same direction.The guards shift collectively to maintain their relative positions in a distance-$k$ dominating set, and occupy the attacked vertex. The resulting distance-$k$ dominating set is equivalent to the original, and the guards can respond to attacks in this manner indefinitely.\qed\end{proof} 

As a consequence of Theorem~\ref{obs:cycle}, whenever a graph $G$ is Hamiltonian, we obtain the following upper bound, by simply considering the guards moving strictly along the Hamilton cycle.

\begin{corollary}\label{cor:Ham} 
Let $G$ be a Hamiltonian graph on $n$ vertices.  Then for $k \geq 1$, $\gk(G) \leq \lceil \frac{n}{2k+1}\rceil$.
\end{corollary}

Although it is easy to see that $\gamma_k(P_n)=\lceil \frac{n}{2k+1}\rceil$, we next show the eternal distance-$k$ domination number of a path is larger. Furthermore, this is an example of a graph $G$ for which $\gk(G) = \gamma_{\lfloor k/2 \rfloor}(G)$.

\begin{theorem}\label{thrm:paths} For $n \geq 1$ and $k \geq 1$, $\gk(P_n)=\lceil \frac{n}{k+1} \rceil$.
\end{theorem}

\begin{proof} Let $P_n = (v_0,v_1,...,v_{n-1})$. We partition the path into vertex-disjoint subpaths, $P_{(i)} = P_n[\{v_{i(k+1)},v_{i(k+1)+1}, ..., v_{(i+1)(k+1)-1}\}]$ for $0\leq i \leq \lfloor \frac{n}{k+1}\rfloor-1$, each of length at most $k$, with the last path being the remaining vertices, $P_{\left( \lfloor \frac{n}{k+1}\rfloor \right)} = P_n[\{v_{\lfloor \frac{n}{k+1}\rfloor}, ... , v_{n-1}\}]$. We then assign one guard to each such subpath, so that $g_i$ is in $P_i$ for all defined paths. It is easy to see that $\lceil \frac{n}{k+1}\rceil$ guards will suffice to form an eternal distance-$k$ dominating family on $P_n$, by the strategy given in  Proposition~\ref{eqn:upper}. To prove the lower bound, we will show that  $\lceil \frac{n}{k+1} \rceil -1$ guards is insufficient. 

Let $P_{(j)}$ be the particular path that does not contain a guard, and denote the set of guards $\{g_0, g_1, .., g_{j-1},g_{j+1}, ..., g_{\lfloor \frac{n}{k+1}\rfloor}\}$ so that each $g_i$ is on $P_{(i)}$. 

Since $v_0$ is a leaf there must always be some guard within distance $k$ of $v_0$, and thus on $P_{(0)}$. The guard $g_0$ cannot leave $P_{(0)}$. Furthermore, no other guard will move into $P_{(0)}$, since $g_0$ can guard all the vertices of $P_{(0)}$. We consider the vertices in $P_{(0)}$ to be ``guarded." Next, in $P_{(1)}$, we make the same argument. Since $g_0$ cannot leave $P_{(0)}$, there must be some other guard within distance $k$ of $v_{1(k+1)}$, and thus $g_1$ must remain within $P_{(1)}$. Suppose some $P_{(j)}$ did not contain a guard. No guard from paths $P_{(0)}, ..., P_{(j+1)}$ can enter $P_{(j)}$, or else they will leave some vertex unguarded. If a guard from $P_{(j+1)}$ enters $P_{(j)}$, the distance from $g_{j+1}$ to the final vertex in $P_{(j+1)}$, $v_{(j+2)(k+1)-1}$ will be more than $k$, thus requiring the guards $g_{j+2},...,g_{\lfloor \frac{n}{k+1}\rfloor}$ to enter paths $P_{(j+1(},...P_{\left( \lfloor \frac{n}{k+1}\rfloor - 1 \right)}$, respectively. Since the distance $\textrm{dist}(v_{(\lfloor \frac{n}{k+1}\rfloor - 1)(k+1)}-1, v_{n-1}) > k$, no guard remains to defend the leaf $v_{n-1}$. In summary, if any path has no guard, when the guards move to defend that path, then they will leave a leaf undefended. So each path requires at least one guard, thus proving the lower bound.\qed\end{proof}

The previous results show that for some families of graphs, the eternal distance-$k$ domination number grows linearly with the order of the graph.  On the other end of the spectrum, we can easily characterize graphs with eternal distance-$k$ domination number $1$:  $\gk(G)=1$ if and only if the diameter of graph $G$ is at most $k$.

For any given $k$, if we fix the $k$-eternal domination number to be $1$, the size of the vertex set, $n$, can take on any value. One interesting result we can obtain from Theorem~\ref{thrm:paths} is the following. Let $P_{n,\ell}$ be $P_{n}$ with $\ell$ leaves added to a vertex adjacent to one of the leaves of $P_{n}$.

\begin{corollary}
For any given positive integers $k,z$ and $n$, with $n\geq z(k+1)$ there exists a graph on $n$ vertices whose eternal distance-$k$ domination number is $z$, namely $P_{z(k+1),\ell}$, where $\ell=n-z(k+1)$.
\end{corollary}

\begin{proof}
For a given $k$ and any positive integer $z$ and any positive integer $n\geq z(k+1)$ consider the path $P_{z(k+1)}$. This has $z(k+1)$ vertices and from Theorem~\ref{thrm:paths} we know that the eternal distance-$k$ domination number is $z$. Place a distance-$k$ dominating set on this graph so that it is eternally distance-$k$ dominated. Label the vertices of the path $v_1,v_2,...,v_{z(k+1)}$, with $v_1$ and $v_{z(k+1)}$ being the leaves. Add $n-z(k+1)$ leaves to vertex $v_{z(k+1)-1}$. The guard that is dominating $v_{z(k+1)}$ also dominates these new leaves. Thus, this new graph has $n$ vertices and eternal distance-$k$ domination number $z$.
\qed\end{proof}

An important question to ask when investigating a graph parameter is, how difficult is it to compute? To answer this, we will look at the relationship between the eternal distance-$k$ domination number of a graph and its graph power. For a graph $G$, the $k^{\textrm{th}}$ power of the graph, $G^k$, is formed by adding an edge $u,v\in E(G)$ whenever $\textrm{dist}(u,v)\leq k$. Thus, if there exists a path in $G$ from $u$ to $v$ of length at most $k$, then we will witness an edge $uv \in E(G^k)$.

\begin{theorem}\label{corr:gkiff}
Let $G$ be a graph and $k\in \mathbb{N}$. Let $S\subseteq V(G)$. Then, $S$ is an eternal distance-$k$ dominating set of $G$ if and only if $S$ is an eternal dominating set of $G^k$.
\end{theorem}

\begin{proof}
Let $S$ be an eternal $k$-dominating set of $G$ for which $|S|\geq \gk(G)$. Place guards on the vertices of $S$ in $G^k$ and suppose a sequence of attacks, $(a_1,a_2,\ldots a_\ell) \subseteq V(G^k)$ occur. 

Consider the same sequence of attacks on $G$ and for each guard $g_i$, let $d_0$ be the vertex $g_i$ initially occupies in $G$ and let $\mathcal{G}_i~=~(d_1,d_2,\ldots d_\ell) \subseteq V(G)$ be the set of corresponding defending moves the guard makes, that is guard $g_i$ moves from $d_{j-1}$ to $d_j$ in $G$ after attack $a_j$. We now consider the eternal 1-domination of $G^k$ by using corresponding moves of eternal distance-$k$ domination of $G$. For any guard $g_i$ and their sequence $\mathcal{G}_i$, moving from $d_{j-1}$ to $d_j$ in $G^k$ is permissible as these vertices have distance at most $k$ in $G$, and thus are adjacent in $G^k$. Each guard $g_i$ can use the same sequence of moves and still guard $G^k$.

Similarly, let $S'$ be an eternal $1$-dominating set of $G^k$ for which $|S'|\geq \gamma_{all,1}^{\infty}(G^k)$. Place guards on the vertices of $S'$ in $G$ and suppose a sequence of attacks $(a_1,a_2,\ldots a_\ell) \subseteq V(G)$ occur. 

Consider the same sequence of attacks on $G^k$ and for each guard $g_i$ in $G^k$, we define $\mathcal{G}_i$ analogously as above. We then consider the eternal distance-$k$ domination of $G$, using moves from the eternal distance-$1$ domination in $G^k$. Any guard $g_i$ that moves from $d_{j-1}$ to $d_j$ after an attack $a_j$ in $G^k$ can also move from $d_{j-1}$ to $d_j$ in $G$ since these two vertices are adjacent in $G^k$ and thus are at most distance $k$ in $G$, hence guard in $G$, giving the desired result.
\qed\end{proof}

\begin{proof}
Let $S$ be an eternal $k$-dominating set of $G$ for which $|S|\geq \gk(G)$. Suppose a sequence of attacks, $(a_1,a_2,\ldots a_\ell) \subseteq V(G)$ occur. For each guard $g_i$, let $d_{i,0}$ be the vertex they are initially placed on and let $\mathcal{G}_i~=~(d_1,d_2,\ldots d_\ell) \subseteq V(G)$ be the set of corresponding defending moves the guard makes, that is guard $g_i$ moves from $d_{j-1}$ to $d_j$ after attack $a_j$.  We now consider the eternal 1-domination of $G^k$ by using corresponding moves of the eternal $k$-domination of $G$. Place the guards in $S$ on the vertices of $G^k$, since $V(G^k)=V(G)$. Suppose in $G^k$ the same sequence of attacks occur on vertices $(a_1,a_2,\ldots a_\ell)$. 

For any guard $g_i$ and their sequence $\mathcal{G}_i$, moving from $d_{j-1}$ to $d_j$ in $G^k$ is permissible as these vertices have distance at most $k$ in $G$, and thus are adjacent in $G^k$. Each guard $g_i$ can use the same sequence of moves and still guard $G^k$.

Similarly, let $S'$ be an eternal $1$-dominating set of $G^k$ for which $|S'|\geq \gamma_{all,1}^{\infty}(G^k)$. Suppose a sequence of attacks $\{a_1,a_2,\ldots a_\ell\}$ occur. For each guard $g_i$, we define $\mathcal{G}_i$ analogously as above. We then consider the eternal $k$-domination of $G$, using moves from the eternal $1$-domination in $G^k$. Place the guards in $S'$ on the vertices of $G$ and consider the eternal $k$-domination process. Again suppose in $G$ the vertices $(a_1,a_2,\ldots a_\ell)$ are attacked in that order. Any guard $g_i$ that moves from $d_{j-1}$ to $d_j$ after an attack $a_j$ in $G^k$ can also move from $d_{j-1}$ to $d_j$ in $G$ since these two vertices are adjacent in $G^k$ and thus are at most distance $k$ in $G$, hence guard in $G$, giving the desired result.
\qed\end{proof}

\begin{corollary} \label{gpower}
If $G$ is a graph and $k \in \mathbb{N}$, then \[\gk(G) = \gamma_{all,1}^{\infty}(G^k).\]
\end{corollary}

In \cite{chipcomplexity} and the subsequent errata \cite{chipwebsite}, it was shown that deciding whether  a set of vertices of a graph is an eternal domination set is in EXP. Thus, taken with Theorem~\ref{corr:gkiff} we have the following complexity result.

\begin{corollary}\label{complexity}
Let $G$ be a graph of order $n$, $k$ be a positive integer with $0\leq k\leq n$, and $S\subseteq V(G)$. Deciding whether $S$ is an eternal distance-$k$ dominating set for $G$ is in EXP.
\end{corollary}

\subsection{Using trees to bound $\gk$}\label{section_superstar}

In this section, we provide insight to the vertices ``guarded'' by a single guard or a pair of guards, and present bounds on the eternal distance-$k$ domination number for arbitrary graphs based on a partitioning into vertex-disjoint trees.

A subgraph $H$ of a graph $G$ is a {\it retract} of $G$ if there is a homomorphism $f$ from $G$ to $H$ so that $f(x)=x$ for $x \in V(H)$.  The map $f$ is called a {\it retraction} and we note that since this is an edge-preserving map to an induced subgraph, the distance between any two vertices does not increase in the image.  

\begin{lemma}\label{lemma:subgraph}Let $H$ be a retract of graph $G$.  Then $\gk(H) \leq \gk(G)$. \end{lemma}  

\begin{proof} Let $H$ be a retract of graph $G$ with the retraction $f:G \rightarrow H$.  We consider two parallel incidences of eternal distance-$k$ domination: one on $G$ and one on $H$. Assume we have a strategy for the guards to defend against any sequence of attacks in $G$, and we will construct a such a strategy in $H$. 

Initially, if there is a guard at vertex $v \in V(G)$, then we place a guard at vertex $f(v) \in V(H)$.  If a guard in $G$ moves from $x$ to $y$ in response to an attack, we observe that guard may move from $f(x)$ to $f(y)$ in $H$ in response to the attack. Regardless of whether $f(v) \in V(H)$ for a particular move of a guard in $H$, the guard in $H$ is either on the attack, or at a distance close enough to the subsequent attack, since retractions cannot increase distances. Thus, $\gk(H) \leq \gk(G)$.\qed\end{proof}

For a tree, we can consider the retraction from $T$ to some subgraph $T' = T[V\backslash D]$ with $D \subseteq V(T)$ by retracting any vertex $v \notin V(T')$ to its nearest neighbour in $V(T')$. This path is unique, and in subsequent sections we will use a shadow strategy for the guards along the retraction as above.

Below we consider another subgraph that will prove useful in obtaining an upper bound for $\gk(G)$ for an arbitrary graph $G$.  Observe that any strategy that protects $G-e$ also protects $G$.

\begin{lemma}\label{Lem:removeedge}
Let $G$ be a graph, $k$ a positive integer and $e\in E(G)$. Then \[\gamma_{all,k}^\infty(G-e)\geq \gamma_{all,k}^\infty(G).\] 
\end{lemma}

\begin{proof} 
Any edge a guard moves along in $G-e$ is also an allowable movement in $G$. Thus, a strategy for the guards to defend against any sequence of attacks in $G-e$ will also suffice for $G$. \qed\end{proof}

From repeated applications of Lemma~\ref{Lem:removeedge} we obtain the following. 

\begin{corollary}\label{cor:span}
Let $G$ be a graph and $T$ a spanning tree of $G$. Then
\[ \gamma_{all,k}^\infty(T)\geq \gamma_{all,k}^\infty(G).\] 
\end{corollary}

Corollary~\ref{cor:span} suggests that understanding the eternal distance-$k$ domination process on trees is important, as it provides an upper bound on the eternal distance-$k$ domination number of a general graph. We next consider covering a graph $G$ with subtrees that have a particular structure, each of which can be guarded by a single guard. By cover, we mean that a guard is assigned to a particular subtree and they can respond to any sequence of attacks that occur within that particular subtree.

\begin{definition} A {\bf depth-$k$-rooted tree} with $k \in \mathbb{Z}^+$, is a rooted tree where the eccentricity of the root is at most $k$.\end{definition}

\begin{definition}  Given a graph $G$, we define a {\bf depth-$k$-rooted tree decomposition} to be a partition of the vertices into sets $S_i$ for $1 \leq i \leq \ell$ such that $G[S_i]$ contains a spanning subgraph that is depth-$k$-rooted tree. 

The {\bf depth-$k$-rooted tree decomposition number} of a graph $G$, denoted $\mathfrak{T}_k(G)$, is the minimum number of sets $S_i$ over all possible decompositions.\end{definition}

An easy bound comes from partitioning graph $G$ into depth-$\lfloor \frac{k}{2}\rfloor$-rooted trees, as one guard defend against a sequence of attacks at vertices of a depth-$\lfloor \frac{k}{2}\rfloor $-rooted tree.  

\begin{corollary}
For any graph $G$ and $k\geq 2$
\[\gk(G) \leq \mathfrak{T}_{\lfloor \frac{k}{2}\rfloor} (G). \]
\end{corollary}

For some graphs, a better bound can be achieved by partitioning graph $G$ into depth-$k$-rooted trees and recognizing that $2$ guards can protect the vertices of each depth-$k$-rooted tree. As a simple example, consider a perfect $m$-ary tree of depth $2$ for $m>2$ (perfect $m$-ary trees are defined in Section~\ref{sec:mary}): the above corollary tells us the eternal distance-$2$ domination number is at most $m$ whereas the next corollary tells us it is at most $2$.

\begin{proposition}\label{propkroot} If $T$ is a depth-$k$-rooted tree for some $k \geq 2$, then $\gk(T) \leq 2$.\end{proposition} 

 \begin{proof}Let $T$ be a depth-$k$-rooted tree with root $r$.  Initially place one guard at $r$ and the second guard at an arbitrary vertex, $u$.  After a vertex $v$ is attacked, the guard at $r$ moves to $v$ and the guard at $u$ moves to $r$.  The resulting distance-$k$ dominating set is equivalent to the original, and the guards can respond to attacks in this manner indefinitely.\qed\end{proof} 
 
 \begin{corollary} For any graph $G$ and $k \geq 2$, $$\gk(G) \leq \min\{ 2\cdot\mathfrak{T}_{k}(G), \mathfrak{T}_{\lfloor k/2 \rfloor} (G) \}.$$ \end{corollary}
 
 In light of Corollary~\ref{cor:span} and the fact that we can partition a graph $G$ into depth-$k$-rooted trees in order to find an upper bound for $\gk(G)$, the next two sections will focus on trees.

\section{Eternal distance-$k$ domination on trees}\label{sec:trees}

In the aim of working towards determining $\gk(T)$ for any tree $T$, we consider the relationship between the eternal distance-$k$ domination number of a tree and particular subtrees. We focus on finding the specific conditions for the subtrees that allow for inductive strategies for $\gk(T)$.

In~\cite{KlosterM}, the authors provide a linear-time algorithm for determining the eternal $1$-domination number of a tree.  Their algorithm consists of repeatedly applying two reductions, {\bf R1} and {\bf R2}, which we restate here. \medskip

\noindent {\bf R1}: Let $x$ be a vertex of $T$ incident to at least two leaves and to exactly one vertex of degree at least two.  Delete all leaves incident to $x$. \medskip

\noindent {\bf R2}: Let $x$ be a vertex of degree two in $T$ such that $x$ is adjacent to exactly one leaf, $y$.  Delete both $x$ and $y$. \medskip

\noindent If $T'$ is the result of applying either {\bf R1} or {\bf R2} to tree $T$, then by \cite{KlosterM} $$\gamma_{all,1}^\infty(T')=\gamma_{all,1}^\infty(T)-1.$$

We generalize the reductions of~\cite{KlosterM} to arbitrary $k\geq 1$, in Proposition~\ref{propR2} for {\bf R2} and Proposition~\ref{propR1} for {\bf R1} to work towards determining the eternal distance-$k$ domination number for trees -- the generalization of {\bf R2} precedes that of {\bf R1} because it is much simpler.   In particular, Propositions~\ref{propR2} and~\ref{propR1} provide substructures whose deletion decreases the eternal distance-$k$ domination number by one. Figure~\ref{fig:R1R2} provides a visualization of the subtrees described in Propositions~\ref{propR2} and~\ref{propR1} for $k=2$.

\begin{figure}[htbp]
 \[ \includegraphics[width=0.55\textwidth]{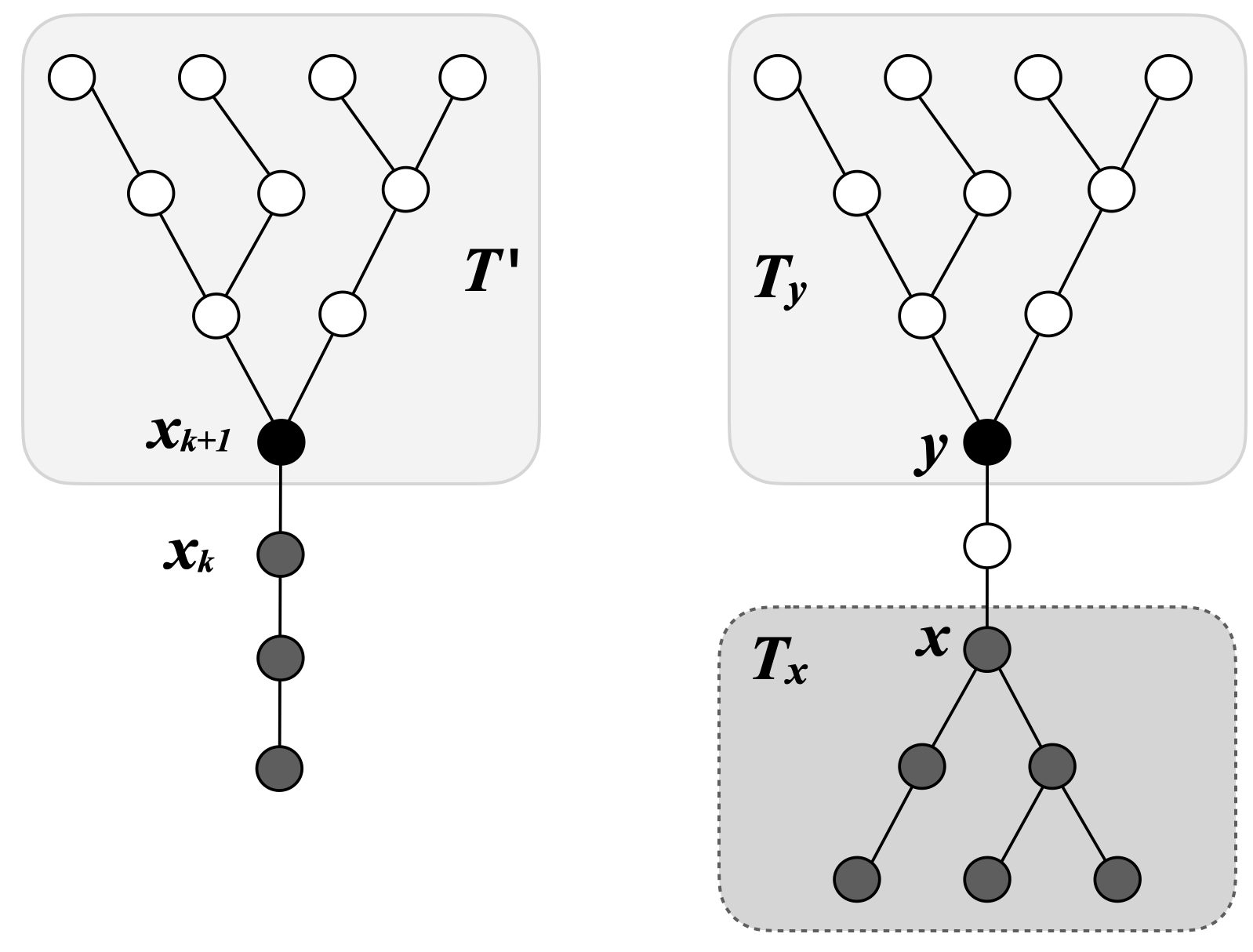} \]
\caption{Examples illustrating the reductions of Propositions~\ref{propR2} and~\ref{propR1} for $k=2$.} 
\label{fig:R1R2}

\end{figure}

 An {\it $i$-end-path} in a graph $G$ is a path of length at least $i \geq 2$ such that at least one endpoint of the path is a leaf and all internal vertices of the path have degree exactly $2$.    

\begin{proposition}\label{propR2} Let $P=(x_0,x_1,\dots,x_{k+1})$ be a $k+1$-endpath of a tree $T$, where $x_0$ is a leaf.  Let $T'=T-\{x_0,x_1,\dots,x_{k}\}$.  Then $\gamma_{all,k}^\infty(T')=\gamma_{all,k}^\infty(T)+1$.\end{proposition} 

\begin{proof} Clearly $\gamma_{all,k}^\infty(T) \leq \gamma_{all,k}^\infty(T')+\gamma_{all,k}^\infty(P-x_{k+1})$ and $\gamma_{all,k}^\infty(P-x_{k+1})=1$ since it is a path of length $k$. We next show that $ \gamma_{all,k}^\infty(T')$ guards do not suffice to eternally $k$-dominate $T$.

First, we consider eternal distance-$k$ domination on the graph $T'$. Place the guards and consider a minimal finite sequence of attacks $\mathcal{A}=(a_1,a_2,\dots,a_q)$ that requires all $\gk(T')$ guards to defend this sequence of attacks. That is, any configuration containing fewer guards cannot defend the attack sequence.

Now consider the graph $T$. Suppose $\gk(T)=\gk(T')$ and place the guards on the vertices of an eternal distance-$k$ dominating set of $T$, note that there is at least one guard on $P$. Consider the sequence of attacks on tree $T$: \[(x_0,a_1,x_0,a_2,x_0,a_3,\dots,x_0,a_q)\] Thus there is always a guard in $P$. This means that guard is never able to defend a vertex in $T'$, hence, there are not enough guards to eternally $k$-dominate $T$, a contradiction.\qed\end{proof}

An {\it $i$-path} in tree $T$ is a path of length $i$ where $i \geq 2$, such that the endpoints of the path do not have degree $2$ and all internal vertices of the path have degree exactly $2$.

\begin{proposition}\label{propR1} Let $P$ be a $k$-path in $T$ with endpoints $x$ and $y$ and denote the subtrees of $T-E(P)$ that contain $x$ and $y$, by $T_x$ and $T_y$, respectively. Denote by $T_y+P$, the tree obtained from $T_y$ by adding the vertices (including $x$) and edges of path $P$. 
If 
\begin{enumerate}
\item $\epsilon_{T_x}(x) = k$,
\item $\epsilon_{T_y}(y) \geq k$, and
\item $diam(T_x)=2k$ 
\end{enumerate}
then $\gamma_{all,k}^\infty(T_y+P) = \gamma_{all,k}^\infty(T) -1$.\end{proposition}

\begin{proof} It is easy to see $\gamma_{all,k}^\infty(T) \leq \gamma_{all,k}^\infty(T_y+P)+1$: initially place $\gamma_{all,k}^\infty(T_y+P)$ guards on an eternal distance-$k$ dominating set of subtree $T_y+P$. Place an additional guard on $x$, or on a neighbour in $T_x$ if one already exists on $x$. Whenever a vertex in $T_x - x$ is attacked,  the guard at $x$ moves to the attacked vertex.  If there was another guard in $T_x - x$, that guard now moves to $x$. The guards in $T_y+P$ move as they would if $x$ were attacked. When a vertex in $T_y +P$ is attacked, the guards in $T_x - x$ move to $x$ or a neighbour of $x$ in $T_x$. The guards remaining in $T_y +P$ defend against the attacked vertex. Thus, there exists a defending strategy for the guards so that after each attack, a guard occupies $x$ and no two guards occupy the same vertex.

We next prove that $\gk(T_y+P) < \gk(T)$.  To do this, we assume $\gk(T)=m$ and show $m-1$ guards can defend against any sequence of attacks on $T_y+P$, that is $\gk(T_y+P) \leq m-1$.

Let $T_x+P$ be the subtree obtained from $T_x$ by adding the vertices (including $y$) and edges of path $P$.  Since $\textrm{diam}(T_x)=2k$ and $\epsilon_{T_x}(x)=k$, let $\ell_1, \ell_2$ be leaves in $T_x$, distance $2k$ from each other, and each distance $k$ from $x$.   Observe that in every eternal distance-$k$ dominating set of $T$, there must always be at least two guards in $T_x+P$ and at least one guard in $T_x$.  Otherwise, if $\ell_1$ is attacked and a guard moves to $\ell_1$, no guard can move to distance $k$ of $\ell_2$.  

Suppose a vertex of subtree $T_y+P$ is attacked.  The $m$ guards can move to form a distance-$k$ dominating set containing the attacked vertex such that there is a guard at vertex $x$ by following a retraction strategy from $T$ as described in Lemma~\ref{lemma:subgraph}. Let $f:V(T) \rightarrow V(T_y+P)$ be a retraction mapping all vertices of $T_x$ to $x$. Whenever a guard in $T$ moves to a vertex $v_x$  in $T_x -x$, the corresponding guard in $T_y +P$ moves to $f(v_x) = x$. From above, we must always have at least one guard in $T_x$. Thus, in the retraction strategy in $T_y+P$, there must always be a guard on $x$. Observe however, that if there is an attack in $T$ on $\ell_1$ then from above, another guard would need to move to $x$ in order to distance-$k$ dominate $\ell_2$. Using the above retraction strategy, this maps two guards to $x$ in $T_y+P$.    

Recall that in $T$, there were always at least two guards in subtree $T_y+P$.  Thus, in the retraction, there are always at least two guards in $T_y+P$.  One guard, however, is superfluous as the vertices of path $P$ can be defended by one guard since $P$ has length $k$. So using the retraction strategy, we form an eternal distance-$k$ dominating set for $T_y+P$ that uses $m-1$ guards.  \qed\end{proof}

The two previous results provided a means to ``trim branches" off a tree to reduce the eternal distance-$k$ domination number in a controlled way. Our next result provides a way to ``trim branches" without changing the eternal distance-$k$ domination number.  For $k=2$, for example, if a vertex is adjacent to at least two leaves, then a leaf can be deleted without changing the eternal distance-$2$ domination number.  We say that the \emph{inclusive-descendants} of a vertex, $v$ in a rooted tree, are all of the descendants of $v$ along with $v$ itself. Figure~\ref{fig:k2rootedtree} shows an example of the tree $T$ and vertices $v_1, ..., v_d$ as defined in Theorem~\ref{thm:sup}.

\begin{figure}[htbp]
    \centering
    \includegraphics[width=0.5\textwidth]{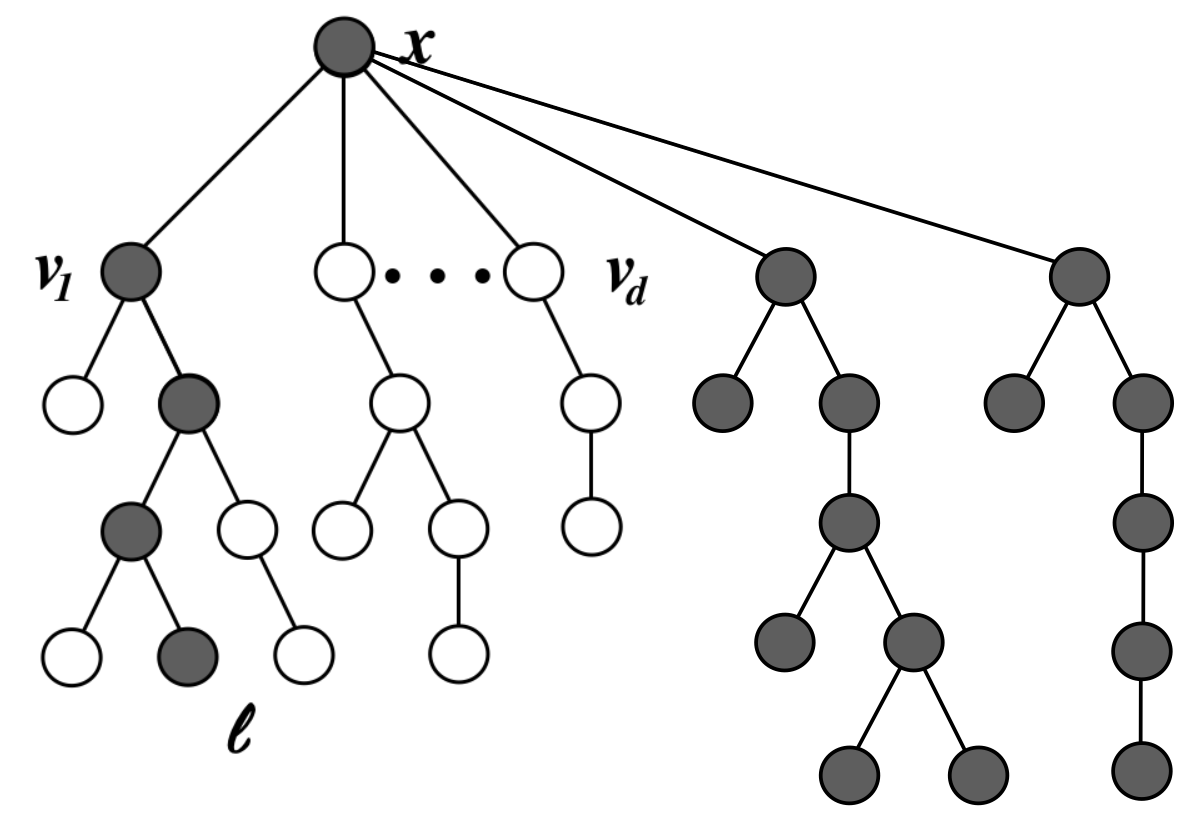}
    \caption{The vertices $x, v_1,$ and $\ell$ from the statement of Theorem~\ref{thm:sup}, where $T_1$ is highlighted in grey and $k=8$.}
   \label{fig:k2rootedtree}
\end{figure}

\begin{theorem}\label{thm:sup} Let $T$ be a tree rooted at $x$.  Suppose $x$ has children $v_1,\dots,v_d$ with the property that all inclusive-descendants of $v_1,\dots,v_d$ are within distance $\lfloor \frac{k}{2}\rfloor$ of $x$ and there is a leaf $\ell$ that is an inclusive-descendant of $v_1$ that is exactly distance $\lfloor \frac{k}{2}\rfloor$ from $x$. Note that $x$ may have other children as well. 

Let $T_1$ be the tree obtained from $T$ by deleting $v_2,\dots,v_d$, all descendants of $v_2,\dots,v_d$, and all descendants of $v_1$ apart from those on the $v_1\ell$-path (or $v_1$ itself if $v_1=\ell$).  Then $\gk(T_1)=\gk(T)$.
\end{theorem}

\begin{proof} By Lemma~\ref{lemma:subgraph}, $\gk(T_1) \leq \gamma_{all,k}^\infty(T)$. It is easy to see that   $\gk(T_1)$ guards suffice to eternally distance-$k$ dominate $T$.  

The guards on $T$ move as they would on graph $T_1$, with one exception: suppose a vertex $v$ of $V(T)\backslash V(T_1)$ is attacked. The guards move as if leaf $\ell$ was attacked, only instead of moving a guard $g$ to $\ell$, guard $g$ instead moves to the attacked vertex $v$.  Since all inclusive-descendants of $v_1,v_2,\dots,v_d$ are within distance $\lfloor \frac{k}{2} \rfloor$ of $x$ (note this includes both $\ell$ and $v$), any guard $g$ that can move to $\ell$, could instead move to $v$.  Finally, any vertex distance-$k$ dominated by $\ell$ is also distance-$k$ dominated by $v$.  \qed\end{proof} 


While the statement of Theorem~\ref{thm:sup} is given for a single tree $T$, if we consider a subgraph $T'$ of a tree $T$, then we may ``trim'' the branches of $T'$ without changing the eternal distance-$k$ domination number of the subtree. Lemma~\ref{lemA} describes another set of vertices that can be deleted from a tree without changing the eternal distance-$k$ domination number.  We first identify two leaves $\ell_1,\ell_2$ that are distance $2k$ apart and let $x$ be the centre of the $\ell_1\ell_2$-path.  Informally, we identify all ``branches'' from $x$ whose leaves are all within distance $k$ of $x$ and remove all vertices on these branches, apart from those on the $\ell_1\ell_2$-path.  Lemma~\ref{lemA} proves the resulting subtree has the same eternal distance-$k$ domination number as the original tree.  An example is shown in Figure~\ref{fig:scribbles}, where the vertices removed are uncolored.

\begin{figure}[ht]
\[\includegraphics[width=0.35\textwidth]{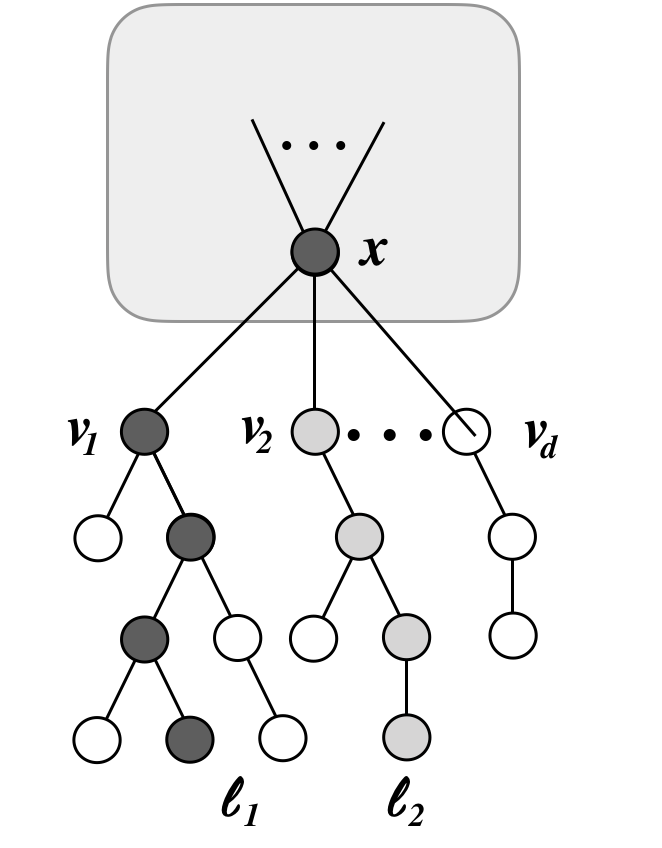}\]
\caption{An example pertaining to Lemma~\ref{lemA} where the $v_1\ell_1$-path is grey, and the $v_2\ell_2$-path is light grey.}

\label{fig:scribbles}

\end{figure}

\begin{lemma}\label{lemA} Let $T$ be a tree rooted at $x$.  Suppose $x$ has children $v_1,\dots,v_d$ with the following properties:
\begin{enumerate}
\item all descendants of $v_1,\dots,v_d$ are within distance $k$ of $x$;
\item there is a leaf descendant $\ell_1$ of $v_1$ that is exactly distance $k$ from $x$, and 
\item there is a leaf descendent $\ell_2$ of $v_2$ that is exactly distance $k$ from $x$.
\end{enumerate}
Note that $x$ may have other children as well. Let $T_1$ be the tree obtained from $T$ by deleting $v_3,\dots,v_d$, all descendants of $v_3,\dots,v_d$, and all descendants of $v_1$ and $v_2$ apart from those on the $v_1\ell_1$-path and $v_2\ell_2$-path.  Then $\gk(T_1)=\gk(T)$.
\end{lemma}

\begin{proof} By Lemma~\ref{lemma:subgraph}, $\gamma_{all,k}^{\infty}(T_1)\leq \gamma_{all,k}^{\infty}(T)$.  Suppose $\gamma_{all,k}^{\infty}(T_1)=m$.  We simply modify the movements of the $m$ guards on $T_1$ to defend against attacks on $T$.  

We first point out the existence of a particular distance-$k$ dominating family on $T_1$. Each eternal distance-$k$ dominating set on $T_1$ must contain at least one vertex on the $x\ell_1$-path (otherwise $\ell_1$ is not distance-$k$ dominated).  Suppose that in response to an attack at a vertex on the  $x\ell_1$-path, the guards move to form a distance-$k$ dominating set that does not contain $x$.  Clearly, the guards could alternately have moved to form a distance-$k$ dominating set that does contain $x$: after the guards move in response to the attack, there must be a guard on the $x\ell_2$-path (otherwise $\ell_2$ is not distance-$k$ dominated) and this guard could have moved to $x$. Thus, there exists an eternal distance-$k$ dominating family on $T_1$ where each eternal distance-$k$ dominating set is of cardinality $m$ and contains $x$.  We now exploit this eternal distance-$k$ dominating family $\mathcal{E}_x$ in order to create an eternal distance-$k$ dominating family for $T$ of cardinality $m$. 

Initially, place $m$ guards on the vertices of $T$ that correspond to the vertices of some eternal distance-$k$ dominating set $S \subseteq \mathcal{E}_x$ on $T_1$.  Since this results in a guard on vertex $x$ on $T$, we know $S$ is a distance-$k$ dominating set on $T$.  If a vertex in subtree $T_1$ is attacked in $T$, the guards in $T$ mirror the movements of how guards $T_1$ would move in response to the attack with one exception:  if a guard in graph $T_1$ moves from vertex $b$ in $T_1$ to $c$ in $T_1$, then the corresponding guard in graph $T$ will move from vertex $b$ to $c$ except in the case that the corresponding guard was in $V(T)\backslash V(T_1)$.  In that case, the guard in $V(T)\backslash V(T_1)$ moves to $x$. Note that in $T$, this results in the vertices of $T$ remaining distance-$k$ dominated. 

Now suppose that on $T$, vertex $z \in V(T)\backslash V(T_1)$ is attacked. If the previous vertex attacked was also in $V(T)\backslash V(T_1)$, then a guard moves from $V(T)\backslash V(T_1)$ to $x$ and a guard moves from $x$ to $z$.  Otherwise, the guards move as if the attack was at $\ell_1$ in $T_1$, except instead of a guard moving from $x$ to $\ell_1$, the guard moves from $x$ to $z$.  This results in the guards occupying a distance-$k$ dominating set in $T$.  \qed\end{proof}

The above lemmas provide some useful tools for reducing the problem of finding the eternal distance-$k$ domination number in trees, however, there are many examples in between where we cannot characterize the change in the eternal distance-$k$ domination number. Given the challenges in relating the eternal distance-$k$ domination number of a tree to a particular subtree, we later focus on the case where $k=2$ in Section~\ref{section_treedecomposition}.

\section{Eternal distance-$k$ domination on perfect $m$-ary trees}\label{sec:mary}

In Proposition~\ref{propd1} and Theorem~\ref{thmd2} we consider the eternal distance-$k$ domination number for perfect $m$-ary trees of small depth.  We then use the results for perfect $m$-ary trees of small depth as the base cases for an inductive proof determining the eternal distance-$k$ domination number of perfect $m$-ary trees of larger depth.

\begin{proposition}\label{propd1} For any $m \geq 2$ and $k\geq 2$, let $T$ be a perfect $m$-ary tree of depth $d\leq k$.  Then $$\gk(T) = \begin{cases} 1 & \text{ if } d \leq \frac{k}{2}\\ 2 & \text{ if } d > \frac{k}{2}. \end{cases}$$  \end{proposition}

\begin{proof} If $d \leq \frac{k}{2}$ then the diameter of $T$ is at most $k$.  Thus, a single guard can successfully defend against any sequence of vertex attacks.  

If $\frac{k}{2} < d \leq k$ then the diameter of $T$ is at least $k+1$.  Thus, the distance-$k$ domination number of $T$ is at least $2$ and therefore $\gk(T) \geq 2$.  

To show the upper bound, we show that two guards can successfully eternally distance-$k$ dominate the tree $T$. First, place one guard on the root $r$ and place another guard at any other arbitrary vertex $u$. Clearly, for all $v \in V(T)$ $d_T(r,v)\leq k$. After any attack, the guard on $r$ moves to the attacked vertex and the guard on $u$ moves to $r$. This is an equivalent configuration to our initial placement, thus we can continue this movement indefinitely.\qed\end{proof}

\begin{theorem}\label{thmd2} For any $m \geq 2$ and $k \geq 2$, let $T$ be a perfect $m$-ary tree of depth $d$ where $k \leq d \leq \frac{3k}{2}$.  Then $$\gk(T) = 1+m^{d-k}. $$\end{theorem}

\begin{proof} When $d=k$, Proposition 2 gives $\gk(T) = 2 =  1+m^{k-k}$ so the statement holds.  Assume $k < d \leq \frac{3k}{2}$ and let $A = \{v_1,v_2,\dots,v_{m^{d-k}}\}$ be the set of vertices that are each distance exactly $k$ from a leaf. Notice that all vertices of $A$ are at depth $d-k \leq \frac{k}{2}$, and thus $d_T(v_i, v_j) \leq k$ for all $v_i, v_j \in A$. This set contains all vertices precisely height $k$ above the leaves, and is well-defined since the depth of this $m$-ary tree is greater than $d$. Consider then, the subtree rooted at $v_i$ for $i \in \{1,\dots,m^{d-k}\}$, and notice that all leaves in this subtree are exactly distance $k$ from $v_i$. We refer to the subtree rooted at $v_i$ as $A_i$ and observe that in every eternal distance-$k$ dominating set, there must be at least one guard in $A_i$; otherwise if a leaf in $A_i$ is attacked, no guard can move to the attacked vertex.  Thus, $\gk(T) \geq m^{d-k}$.  Further, observe that if $\gk(T)= m^{d-k}$ and the root is attacked, no guard can move to the root as a guard in $A_i$ must remain in $A_i$ in the event the subsequent attack occurs at a leaf in $A_i$. Thus, $\gk(T) > m^{d-k}$. 

\begin{figure}[htbp]
    \centering
    \includegraphics[width=0.34\textwidth]{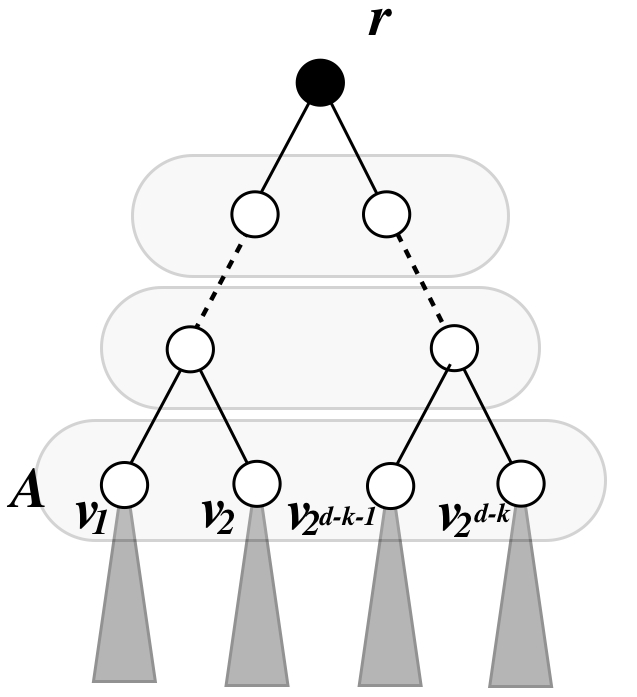}
    \caption{An example of a tree $T$ rooted at $r$ with $m=2$, and the vertices $A$ at depth $d-k$. The subtrees of each $v_i \in A$ are shown as grey shaded triangles.}
    \label{Fig:example}
\end{figure}

To show that $\gk(T) \leq 1+m^{d-k}$, initially place a guard on each vertex in $A$ and place one additional guard at an arbitrary leaf; without loss of generality, suppose this leaf is in $A_1$, the subtree rooted at $v_1$.  The resulting set of guards is distance-$k$ dominating, and we next prove that it is an eternal distance-$k$ dominating set.

If a vertex in $A_1$ is attacked, one guard at $v_1$ moves to the attacked vertex and the other guard  moves to $v_1$ as in the proof of Proposition~\ref{propd1}. If a vertex $A_i$ is attacked for $i >1$, the guard at $v_i$ moves to the attacked vertex, and the guard on $v_1$ can move to $v_i$ since $d_T(v_1, v_i) \leq k$. The additional guard in $A_1$ then moves to $v_1$. The resulting distance-$k$ dominating set is equivalent to our initial set of all guards on $A$ with an additional guard in an arbitrary subtree.

We next show that we can defend against attacks at depth less than $d-k$. Let $T_{d-k-1}$ be the subtree of $T$ rooted at $r$ ,of depth $d-k-1 < \frac{k}{2}$, and note that this subtree may be the single vertex $r$ if $d-k-1=0$.

Suppose the next attack occurs at a vertex $u$ in the tree $T_{d-k-1}$ and the guard who is not on a vertex in $A$ is in the subtree rooted at $v_i$. The guard at $v_i$ moves to the attacked vertex $u$ since $\textrm{diam}(T_{d-k-1}) < k$,  and the guard in subtree $v_i$ moves to $v_i$.  

If the subsequent attack is at a vertex $w \in V(T_{d-k-1})$, the guard at $u$ moves to $w$ and no other guard moves. The resulting distance-$k$ dominating set is equivalent to the previous set of all guards on vertices in $A$ with an additional guard in $T_{d-k-1}$.

Otherwise, if the subsequent attack is at a vertex in $A_t$, for some $t \in \{1,2,\dots,m^{d-k}\}$, the guard at $v_t$ moves to the attacked vertex while the guard at $u$ moves to $v_t$. The resulting distance-$k$ dominating set is equivalent to our initial set of all guards on $A$ with an additional guard in an arbitrary subtree.

We summarize the movement of guards in the following table, where the attack occurs at some vertex $a$.  In the table, $g^\star$ refers to the guard not currently on a vertex of $A$.  Observe that this implies guard $g^\star$ either occupies a vertex of $T_{d-k-1}$ or a vertex in $A_g \backslash \{v_g\}$; additionally vertex $v_a \in A_a$. \qed\end{proof}

\begin{table}[htbp]
\renewcommand{\arraystretch}{1.5}
    \centering
    \begin{tabular}{c|c|c}
         ~ & $a \in A_a$ & $a \in V(T_{d-k-1})$ \\
         \hline
         ~ & $g^\star \rightarrow v_g$ & $g^\star \rightarrow v_g$  \\ 
        $g^\star$ on a vertex in $A_g\backslash \{v_g\}$ & $v_g \rightarrow v_a $  & $v_g \rightarrow a$\\
        
        ~ & $v_a \rightarrow a $ & ~ \\
        \hline
        $g^\star$ on a vertex of $V(T_{d-k-1}) $ & $g^\star \rightarrow v_a$ & $g^\star \rightarrow a$ \\
        ~  & $v_a \rightarrow a$ & ~
    \end{tabular}
    \caption{Summary of the movement of the guards between the two general distance-$k$ dominating sets as described in Theorem~\ref{thmd2}.}
    \label{tab:my_label}
\end{table}

We next determine the eternal distance-$k$ domination number for any perfect $m$-ary tree of depth $d > \frac{3k}{2}$.  The proof applies Proposition 4 repeatedly until what remains is a perfect $m$-ary tree of small depth, allowing for a recursive result.  The results of Proposition~\ref{propd1} and Theorem~\ref{thmd2} will form the base cases for the inductive proof of Theorem~\ref{thmd3}.

\begin{theorem}\label{thmd3} For any $m \geq 2$ and $k \geq 2$, let $T$ be a perfect $m$-ary tree rooted at $r$ of depth $d > \frac{3k}{2}$.  Let $q$ be the unique integer such that $d \equiv q ~(mod ~k$) with $\frac{k}{2} \leq q < \frac{3k}{2}$. Define $T_q$ to be a perfect $m$-ary sub-tree of $T$, rooted at $r$ with depth $q$ and $m \geq 2$ Then $$\gk(T) = \frac{m^d-m^q}{m^k-1}+\gk(T_q).$$ \end{theorem}

\begin{proof} The proof is inductive and uses the results of Proposition~\ref{propd1} and Theorem~\ref{thmd2} as base cases.  Let $T:=T_d$.  One at a time, identify each vertex of $T_d$ at depth $d-k$ as vertex $x$ in Proposition 4, with parent $y$, and apply Proposition 4.  Let $T_{d-k}$ be the resulting perfect $m$-ary tree of depth $d-k$ and by Proposition 4, $\gk(T_d) = \gk(T_{d-k})+m^{d-k}$. Repeat the process by identifying all vertices at depth $d-ik$ if they exist for each $i \in \{1,2,\dots,\frac{d-q}{k}\}$ and applying Proposition 4 until the resulting perfect $m$-ary tree is of depth $q$.  This tree is a sub-tree rooted at $r$ of $T_d$ of depth $q$ and is $T_q$ as defined in the statement of the theorem.

In each iteration of applying Proposition 4, we note that the eccentricity of the vertex currently identified as $y$ must always be at least $k$, since we have a path from $y$ to some leaf in $T_q$ by traversing from $y$ to $r$ and back down to another leaf (since $m \geq 2$), and this path must be exactly length $2q \geq \frac{2k}{2} $ by definition of $q$. This satisfies condition 2. in the proposition statement. Further, conditions 1. and 3. are satisfied by the choice of $x$. So our repeated application of Proposition 4 is well-defined.

Thus, by Proposition 4, $$\gk(T_d) =\Big(\sum_{i=1}^{\frac{d-q}{k}} m^{d-ik}\Big)+\gk(T_q) = \frac{m^d-m^q}{m^k-1}+\gk(T_q).$$ \qed\end{proof}

The above theorem yields recursive results. However, earlier results give us distinct values for $\gk(T_q)$, and a closed form for the eternal distance-$k$ domination number of perfect $m$-ary trees can be determined. We summarize these results in the following corollary.

\begin{corollary}For any $m \geq 2$ and $k \geq 2$, let $T$ be a perfect $m$-ary tree of depth $d$ and $q$ be the unique integer such that $d \equiv q$ (mod $k$) with $\frac{k}{2} \leq q < \frac{3k}{2}$.   Then \[\gk(T) = \begin{cases} 1 & \text{ if } 0 \leq d \leq \frac{k}{2}; \\[0.25cm] 2 & \text{ if } \frac{k}{2} < d < k; \\[0.25cm] 1+m^{d-k} & \text{ if } k \leq d \leq \frac{3k}{2}; \\[0.5cm] \displaystyle 2 + \frac{m^d -m^q}{m^k-1}  & \text{ if } d > \frac{3k}{2} \text{ and } \frac{k}{2} \leq q \leq k; \\[0.5cm] \displaystyle 1+\frac{m^d-m^{q-k}}{m^k-1}& \text{ if } d > \frac{3k}{2} \text{ and } k < q \leq \frac{3k}{2}. \end{cases}\]\end{corollary}

\section{Tree Reductions for eternal distance-$2$ domination}  \label{section_treedecomposition}

The results in Section~\ref{sec:trees} provide reductions for trees that control the change in eternal distance-$k$ domination number.  In this section, we restrict ourselves to $k=2$.  We show that for trees with certain structure, deleting a portion of the tree results in the eternal distance-$2$ domination number decreasing by $1$.  

For the duration of this section, we require the following particular description of tree structure.

\begin{definition}\label{def:allthesets}
Given a tree $T$, let $x \in V(T)$ be a vertex that is not a leaf and is distance exactly two from some leaf.  Define the following sets: 
\begin{itemize}
    \item[$\bullet$] $L = \{v \in V(T) | \textrm{deg}_{T}(v) = 1 \; \textrm{and} \; d_T(v,x) = 2\}$: the set of leaves in $N_2(x)$\\
    \item[$\bullet$] $\displaystyle{X = \bigcup_{\ell \in L} N_2(\ell):}$ the set of all vertices distance two from vertices of $L$\\
    \item[$\bullet$] $S$ be the set of vertices adjacent to $x$ that are either leaves themselves or adjacent to a vertex in $L$. We further partition $S$ into two sets:\\
    \begin{itemize}
   \item[$\bullet$]  $A \subset S$ the set of vertices with at least two neighbours in $X$ 
   \item[$\bullet$] $B \subseteq S$ is the set of vertices with exactly one neighbour in $X$. 
    \end{itemize}
\end{itemize}
\end{definition}

For the above sets, it is useful to note that $L \cap X$ need not be empty. Furthermore, $x \in X$ and thus $A \cup B = S$ since all vertices in $S$ must be adjacent to at least one vertex in $X$. As seen in Figure~\ref{fig:treeillustration} the sets $L,S,A$ and $B$ are all dependent on  vertex $x$, and we can consider these sets defined for any eligible vertex $x \in V(T)$.  Though $L,S,A,B$ each depend on the choice of $x$, we omit any subscripts in an aim to present results and definitions in a more readable way.

\begin{figure}[htbp]
\[\includegraphics[width=0.6\textwidth]{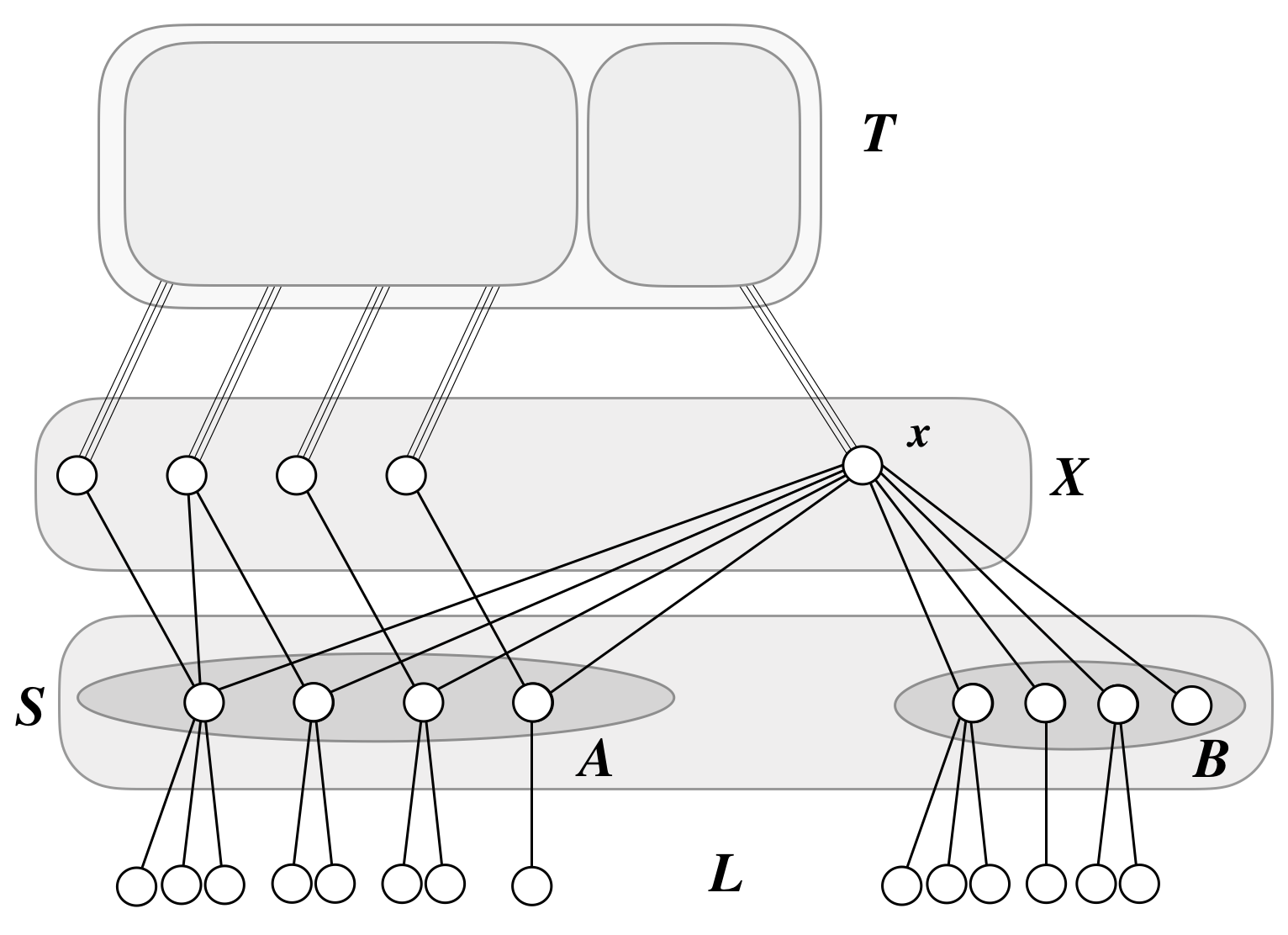} \] 
\caption{An example illustrating the sets defined in Definition~\ref{def:allthesets}.}
\label{fig:treeillustration} 
\end{figure}

\begin{theorem}\label{thm:LASBX} Let $T$ be a tree and suppose there exists a vertex $x \in V(T)$ that generates sets $L$ and $S$ as defined in Definition~\ref{def:allthesets}.  Then $$\gtwo(T) \in \Big\{\gtwo(T'), \gtwo(T')+1\Big\}$$ where $T'$ is the graph induced by the deletion of $L\cup S$ from $T$.\end{theorem} 

\begin{proof} By Lemma~\ref{lemma:subgraph}, $\gtwo(T) \geq \gtwo(T')$. It is easy to see $\gtwo(T) \leq \gtwo(T')+1$: we create a guard strategy for $T$, based on the movements of guards in $T'$.  Initially, suppose guards occupy an eternal distance-$2$ dominating set on $T'$ that contains $x$.  Place guards on the vertices of $V(T')$ in $T$ and place an additional guard on an arbitrary vertex of $L \cup S$.  Suppose there is an attack in $T$ on a vertex in subtree $T'$: if there is a guard in $L \cup S$, that guard moves to $x$. The remaining guards move as their counterparts in $T'$ would move in response to such an attack (this may result in two guards simultaneously occupying $x$).  Alternately, suppose there is an attack at a vertex of $L \cup S$ in $T$.  In $T'$, we consider an attack at $x$.  The guards in $T'$ move to occupy an eternal distance-$2$ dominating set containing $x$ (which may result in no guard moving).  In $T$, the guard at $x$ moves to the attacked vertex in $L \cup S$; if there is a guard already in $L \cup S$, that guard moves to $x$, and the remaining guards move like their counterparts in $T'$.  Thus, the guards form a distance-$2$ dominating set on $T$. In this manner, $\gk(T')+1$ guards can defend against any sequence of attacks on $V(T)$.\qed\end{proof} 

Observe that even if we restrict $A = \emptyset$, there exist trees for which removing set $S \cup L$ will reduce the eternal distance-$2$ domination number, and others which will leave the eternal distance-$2$ domination number unchanged;  two such trees are illustrated in Figure~\ref{fig:remark}.

\begin{figure}[htbp]
\centering
\includegraphics[width=0.75\textwidth]{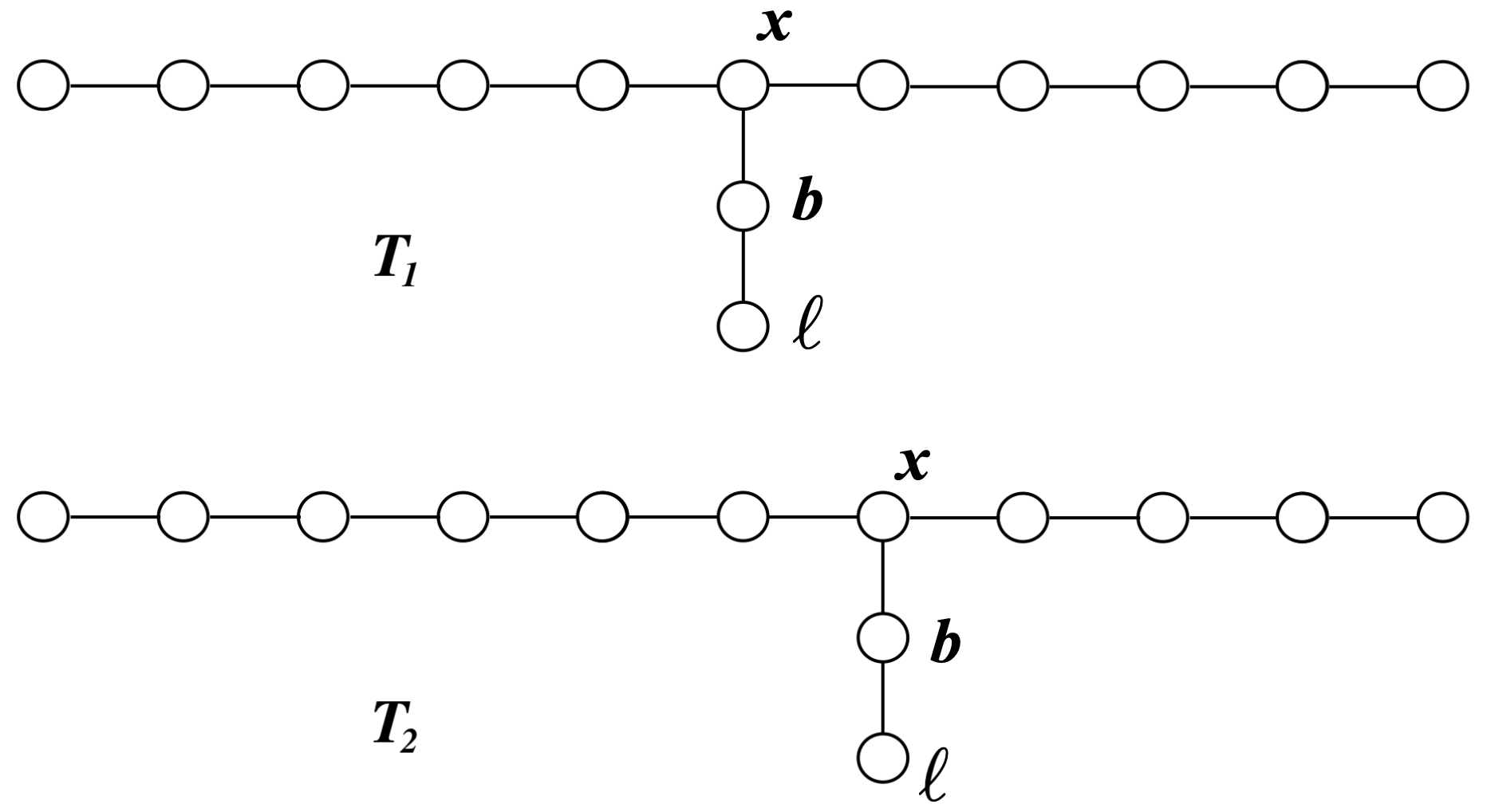}
\caption{Trees $T_1$ and $T_2$.}
\label{fig:remark} 
\end{figure}

Thus, the characterization of $\gtwo(T)$ for all trees remains incomplete and we leave this as an open problem.  In Section~\ref{section_Conclusion} we present further open problems and concluding remarks.

\section{Conclusion and Open Problems}\label{section_Conclusion}

We conclude this paper with a discussion of open problems. Section~\ref{sec:prelim} discussed some graphs for which the parameters $\gamma_k$ and $\gk$ are equal, and others for which the two parameters are not equal; however the question of for which graphs the parameters equal remains open, giving our first open problem.

\begin{question}\label{q1} 
For what class of graphs, $\mathcal{G}$, is $\gamma_k (G) = \gk(G)$ for all $G \in \mathcal{G}$? \end{question}

Clearly there are many classes of graphs for which the answer to Question 1 is no, which leads to our next two open problems.

\begin{question}
	Let $\mathcal{G}_{n,m}$ be the family of simple graphs on $n$ vertices and $m$ edges. For a fixed $n$ and $m$, what are the graphs with the smallest eternal distance-$k$ domination number, or largest eternal distance-$k$ domination number?
\end{question}

\begin{question} Given a fixed $n$ and fixed $k$, what possible values can $\gk(G)$ take on, for graphs $G$ of order $n$?
\end{question}

Though we provide reductions in working towards determining $\gk(T)$ and further reductions in working towards determining $\gtwo(T)$ for any tree $T$, have not yet completed the characterizations and leave them as open problems.  We conclude by proposing a few related questions one could investigate.

\begin{question} Which graphs $G$ have the property that $\gtwo (G) = \gamma(G)$? Can we characterize the trees with this property? \end{question} 

Clearly $\gtwo(G) = 1$ if and only if $\gamma(G)=1$ (i.e. $G$ has a universal vertex).  If $\gamma(G) = 2$ then $\gtwo(G)=2$, but the converse is not always true.  For example, $\gtwo(C_{10})=2 < 4 = \gamma(C_{10})$. In considering trees, we observe that $\gtwo(K_{1,n})=\gamma(K_{1,n})$ but for caterpillar graphs where there are no degree $2$ vertices, $\gtwo$ and $\gamma$ are not equal. However, for every tree $T$ that is formed from a path of $n$ vertices where every internal vertex has a leaf, with $n\geq 3$, we have that $\gtwo (T) \neq \gamma(T)$. 

Certainly if there exists a minimum dominating st where each vertex in this set has at least two private neighbours then $\gtwo(T)\leq \gamma(T)$, but the converse remains open.

\begin{question}\label{q2} Suppose that for every minimum dominating set of a tree $T$, each vertex in the dominating set has at least two private neighbours. Then is $\gtwo(T)=\gamma(T)$? \end{question}

\begin{acknowledgements}
M.E. Messinger acknowledges research support from NSERC (grant application 2018-04059).  D. Cox acknowledges research support from NSERC (2017-04401) and Mount Saint Vincent University.  E. Meger acknowledges research support from Universit\'e du Qu\'ebec \`a Montr\'eal and Mount Allison University.
\end{acknowledgements}

%
%

\end{document}